\documentclass[oneside,english]{article}
\usepackage[T1]{fontenc}
\usepackage{amsthm}
\usepackage{amssymb}
\usepackage[all]{xy}
\usepackage{amsmath}
\usepackage{comment}
\usepackage{color}
\usepackage[nocompress]{cite}

\makeatletter
\newcommand{\nc}{\newcommand}
\newtheorem{thm}{Theorem}
\theoremstyle{plain}
\nc{\bthm}{\begin{thm}} \nc{\ethm}{\end{thm}}
\newtheorem{prop}[thm]{Proposition}
\nc{\bprp}{\begin{prop}} \nc{\eprp}{\end{prop}}
\newtheorem{fact}[thm]{Fact}
\nc{\bfct}{\begin{fact}} \nc{\efct}{\end{fact}}
\newtheorem{prob}[thm]{Problem}
\nc{\bprb}{\begin{prob}} \nc{\eprb}{\end{prob}}
\newtheorem{lem}[thm]{Lemma}
\nc{\blem}{\begin{lem}} \nc{\elem}{\end{lem}}
\newtheorem{claim}[thm]{Claim}
\nc{\bclm}{\begin{claim}} \nc{\eclm}{\end{claim}}
\newtheorem{cor}[thm]{Corollary}
\nc{\bcor}{\begin{cor}} \nc{\ecor}{\end{cor}}
\newtheorem{conj}[thm]{Conjecture}
\nc{\bcnj}{\begin{conj}} \nc{\ecnj}{\end{conj}}
\theoremstyle{definition}
\newtheorem{defn}[thm]{Definition}
\nc{\bdfn}{\begin{defn}} \nc{\edfn}{\end{defn}}
\newtheorem{observation}[thm]{Observation}
\nc{\bobs}{\begin{observation}} \nc{\eobs}{\end{observation}}
\theoremstyle{remark}
\newtheorem{rem}[thm]{Remark}
\nc{\brem}{\begin{rem}} \nc{\erem}{\end{rem}}
\newtheorem{cnv}[thm]{Convention}
\nc{\bcnv}{\begin{cnv}} \nc{\ecnv}{\end{cnv}}
\newtheorem{exam}[thm]{Example}
\nc{\bexm}{\begin{exam}} \nc{\eexm}{\end{exam}}

\newtheorem{question}[thm]{Questuion}
\nc{\bpf}{\begin{proof}} \nc{\epf}{\end{proof}}
\nc{\be}{\begin{enumerate}}
	\nc{\ee}{\end{enumerate}}
\nc{\bi}{\begin{itemize}}
	\nc{\itm}{\item}
	\nc{\ei}{\end{itemize}}
\nc{\invlim}{\lim_{\leftarrow}}
\nc{\dirlim}{\lim_{\rightarrow}}
\nc{\mm}{\mathbf{m}}
\nc{\nn}{\mathbf{n}}
\nc{\FF}{\mathcal{F}}
\nc{\CC}{\mathcal{C}}

\title{Profinite Completion of Free Pro-$\mathcal{C}$ Groups}
\author{Tamar Bar-On}
\date{\today}
\begin{document}
	
	\maketitle
	
	\begin{abstract}
		We investigate the ability of a free pro-$\CC$ group of infinite rank to abstractly solve abstract embedding problems, and conclude that for some varieties $\CC$, the profinite completion of any order, of a free pro-$\CC$ group of infinite rank, is a free pro-$\CC$ group as well, of the corresponding rank. 
	\end{abstract}
\section*{Introduction}
During the paper, $\mm$ will be an infinite cardinal, and $\CC$ be a variety of finite groups, i.e, a collection of finite groups which is closed under taking subgroups, quotients and finite products. The free pro-$\CC$ group of rank $\mm$ is defined as follows:

Consider a discrete topological space of cardinality $\mm$, $X$. Denote by $X\cup \{\ast\}$ its one-point compactification. One easily checks that this is a pointed profinite space- i.e, a topological space which is compact, Hausdorff and totally disconnected, with $\ast$ being the distinguished point. The free pro-$\CC$ group over $X\cup \{\ast\}$, also known as the free pro-$\CC$ group over the set $X$ converging to 1, is a pro-$\CC$ group $F_{\CC}(\mm)$, considered as a pointed profinite space with the identity element serving as the distinguished point, together with a pointed continuous function $i:X\cup \{\ast\}\to F_{\CC}(\mm)$ such that the following universal property is satisfied:

For every continuous pointed map $f:X\cup \{\ast\}\to G$ where $G$ is any pro-$\CC$ group considered as a pointed space in the same way, there exists a unique continuous homomorphism $\bar{f}:F_{\CC}(\mm)\to G$ which makes the following diagram commutative:
	\[
\xymatrix@R=14pt{ F_{\CC}(\mm)\ \ar[rd]^{\bar{f}}& \\
	X\cup \{\ast\}\ \ar[u]_{i} \ \ar[r]_f& G\\
}
\]
It is not hard to prove that the function $i$ is one-to-one, so one can think of $X$ as a subset of $F_{\CC}(\mm)$. The set $X$ is called a basis for $F$, and it serves as a set of (topological) generators \textit{converging to 1}- meaning that every open subgroup of $F_{\CC}(\mm)$ contains all but finite number of elements of $X$. Sets of generators converging to 1 play an important role in profinite groups as every profinite group admits such a set (see \cite[Proposition 2.4.4]{ribes2000profinite}), and its cardinality is an invariant of the group.

It is easy to see that the definition of $F_{\CC}(\mm)$ doesn't depend on the choice of $X$.

 Despite of its name, the free pro-$\CC$ group of rank $\mm$ is not a free object in categorical sense, since $X\cup \{\ast\}$ is not just a set. However, it serves as an initial object in the subcategory of pro-$\CC$ groups of local weight $\mm$. Here, the local weight of a profinite group $G$, denoted by $\omega_0(G)$, stands for the cardinality of the set of its open (normal) subgroups, which is also equal to the cardinality of every set of generators converging to 1. As such, it has a great importance in the studying of profinite groups. Free profinite groups of infinite rank also play an important rule in Galois theory: the free profinite group of rank $\mm$ is the absolute Galois group of the field $k(t)$ where $k$ is some algebraically closed field of cardinality $\mm$ (\cite{douady1964determination}). 
A special case is devoted to free pro-$p$ groups. These are exactly the pro-$p$ groups for which $cd(G)\leq 1$, and their rank is known to be equal to $\dim_{\mathbb{F}_p}(H^1(G,\mathbb{F}_p))$. In fact, they are the building blocks for \textit{projective} groups: a profinite group is projective if and if all its $p$- Sylow subgroups are free pro-$p$ (See \cite[Proposition 7.6.7, Theorem 7.7.4 ]{ribes2000profinite}). A famous conjecture of  L. Positselski (\cite{positselski2005koszul}) states that the maximal pro-$p$ Galois group of $\sqrt[p^{\infty}]{\mathbb{F}}=\mathbb{F}(\sqrt[p^n]{a},a\in \mathbb{F},n\geq 1)$, where $\mathbb{F}$ is a field containing a primitive $p$th root of unity, is a free pro-$p$ group.   

For more information about free pro-$\CC$ groups one may look at \cite[Chapter 3]{ribes2000profinite}, \cite[Chapter 7]{fried2006field} and the various papers that have been written on free pro-$\CC$ groups, such as \cite{jarden1995free, mel1978normal,  jarden1980normal, lubotzky1981subgroups}.

Contrary to the abstract case, for which we have a pretty good understanding, free pro-$\CC$ groups possess a rich, complex and mysterious structure. For example, while open subgroups of free pro-$\CC$ groups of infinite rank are free pro-$\CC$ of the same rank, this is not true in general for closed subgroups. This fact makes the search for free closed subgroups of free pro-$\CC$ groups an interesting subject of research (see, for example, \cite[Chapter 8]{ribes2000profinite} , and \cite{shusterman2016free}). Moreover, while a countable inverse limit of f.g pro-$\CC$ groups is countably free, a general inverse limit of free pro-$\CC$ groups may not be free (\cite{ribes2000profinite} Example 9.1.14). It is still an open question, whether a general inverse limit of f.g free pro-$\CC$ groups is free, necessarily of infinite rank (\cite{melnikov1980projective}).

 In this paper we would like to study the  profinite completion of a free pro-$\CC$ group of infinite rank, which leads us to questions about the abstract structure of a free pro-$\CC$ group, and the lattice of its finite abstract quotients.
 \begin{defn}

 Let $G$ be an abstract group. The profinite completion of $G$, denoted by $\hat{G}$, is a profinite group equipped with a homomorphism $i:G\to \hat{G}$ which satisfies the following universal property: for every homomorphism $f:G\to H$ where $H$ is a profinite group, there exits a unique continuous homomorphism $\hat{f}:\hat{G}\to H$ which makes the following diagram commutative:
 \[
 \xymatrix@R=14pt{ \hat{G}\ \ar[rd]^{\hat{f}}& \\
 	G\ \ar[u]_{i} \ \ar[r]_f& H\\
 }
 \]
  \end{defn}
In fact, the profinite completion of $G$ is the inverse limit of the inverse system of its finite quotients with the natural epimorphisms between them, when $i:G\to \hat{G}$ is the homomorphism induced from the natural projections $G\to G/H$. One easily sees that $i[G]$ is dense in $\hat{G}$, and that $i$ is one-to-one if and only if $G$ is residually finite. 

Now let $G$ be a profinite group. Considered as an abstract group, $G$ also admits a profinite completion. Despite what the first intuition might be, a profinite group may not be equal to its profinite completion. In fact, by the explicit construction, a profinite group is equal to its profinite completion if and only if every subgroup of finite index is open. The following is an important example to a profinite group with infinitely many dense proper subgroups of finite index:
\begin{exam} (\cite[Example 4.2.12]{ribes2000profinite})\label{main example}
Let $T$ be a nontrivial finite group and look at $G=\prod_{\mm}T$ an infinite power of $T$. For every ultrafilter $\FF$ on $\mm$, let $T_{\FF}$ be the subgroup of $G$ consisting of all tuples $(t_i)$ such that $\{i:t_i=1\}\in \mathcal{F}$. One easily sees that $G/T_{\FF}\cong T$, and if $\FF$ is not principal, then $T_{\FF}$ is dense. 
\end{exam}
A profinite group which is equal to its own profinite completion is called \textit{strongly complete}. In Lemma \ref{basic lemma} it is shown that a free pro-$\CC$ group of infinite rank is never strongly complete. Thus, $\widehat{F_{\mathcal{C}}(\mm)}$ is a new profinite group. In addition to the profinite completion we just define, which we can treat as the "first" profinite completion, a nonstrongly complete profinite group admits a series of profinite completions of higher orders, which we denote by $G_{\alpha}$, and are defined for every ordinal $\alpha$. These completions of higher order were first introduced in \cite{BarOn2018tower} and are built as follows: Let $G$ be a nonstrongly complete profinite group. Denote $G=G_0$. If $\alpha=\beta+1$ then $G_{\alpha}=\widehat{G_{\beta}}$. If $\alpha$ is a limit ordinal, $G_{\alpha}=\widehat{H_{\alpha}}$ where $H_{\alpha}={\dirlim}_{\beta<\alpha}\{G_{\beta}, \varphi_{\gamma\beta}\}$, where $\varphi_{\gamma\beta}$ are compatible homomorphisms which are defined for every $\gamma<\beta$.  In \cite{BarOn2018tower} it is shown that all the groups in this chain are nonstrongly complete, and all the homomorphisms are injections. That makes the chain a strictly increasing chain, which never terminates. In \cite{BarOn+2021} it was shown in addition that for every ordinal $\alpha$, if $\alpha=\beta+1$ then $\omega_0(G_{\alpha})=2^{2^{\omega_0(G_{\beta})}}$, and if $\alpha$ is a limit then $ \omega_0(G_{\alpha})={2^{\sup_{\beta<\alpha}{\omega_0(G_{\beta})}}}$.

This paper is going to deal with the following question:

\begin{question}[Main Question]
Is the profinite completion of any order, of a free pro-$\CC$ group of infinite rank, a free pro-$\CC$ group as well? 
\end{question}
This question is motivated by the abstract case: In Lemma \ref{abstract case} it is shown that the profinite completion of a free abstract group is a free profinite group, and in the pro-$p$ case, which was proven by the author in \cite{bar-on_2021} and \cite{baron-cohomological}

The paper is organized as follows: In Section 1, we prove that the Main Question is equivalent to the ability of a free pro-$\CC$ group to solve abstract embedding problems. In Section 2, we give positive results for some special varieties: the varieties of finite $p$-groups, finite nilpotent groups and finite abelian groups. In Section 3 we give some partial results for the general case, which are valid for every variety $\CC$. In Section 4 we prove some results about the existence of noncontinuous homomorphisms from a free pro-$\CC$ group to itself.
\section*{Abstract embedding problems}
We start the paper by proving that: 
\begin{lem} \label{basic lemma}
$F_{\CC}(\mm)$ is not strongly complete		
\end{lem}
\begin{proof}
By \cite[Theorem 3.3.16]{ribes2000profinite}, $F_{\CC}(\mm)$ projects on every pro-$\CC$ group of local weight $\mm$. Let $C\in \CC$ be a finite group. Then $F_{\CC}(\mm)$ projects on $\prod_{\mm}C$. By Example \ref{main example} $\prod_{\mm}C$ admits a proper finite index dense subgroup, and so does $F_{\CC}(\mm)$. In fact, there are $2^{2^{\mm}}$ different nonprincipal ultrafilters over a set of cardinality $\mm$ (see \cite[Section 111]{steen1978counterexamples}), each one of them created a different finite index dense proper subgroup, so $F_{\CC}(\mm)$ admits at least $2^{2^{\mm}}$ subgroups of finite index which are not open.
\end{proof}
One of the useful ways to characterize free pro-$\CC$ groups, is due to \textit{finite} embedding problems:
\begin{prop}(\cite[Proposition 3.5.11]{ribes2000profinite})
Let $G$ be a pro-$\CC$ group with $\omega_0(G)=\mm$. Then $G$ is a free pro-$\CC$ group of rank $\mm$ if and only if every embedding problem of the form:
	\[
\xymatrix@R=14pt{ & & &G \ar@{->>}[d]^{\varphi}& \\
	1 \ar[r] & K\ar[r] & C \ar[r]^{\alpha}&D \ar[r]&1\\
}
\]
where $K\ne 1$ and  $C$ is a finite group belonging to $\CC$, has $\mm$ different solutions.
\end{prop}
In fact, we can refine this condition a bit:
\begin{prop}\label{minimal normal}
	Let $G$ be a pro-$\CC$ group with $\omega_0(G)=\mm$. Then $G$ is a free pro-$\CC$ group of rank $\mm$ if and only if every embedding problem of the form:
	\[
	\xymatrix@R=14pt{ & & &G \ar@{->>}[d]^{\varphi}& \\
		1 \ar[r] & K\ar[r] & C \ar[r]^{\alpha}&D \ar[r]&1\\
	}
	\]
	where $K\ne 1$ is minimal normal non trivial subgroup of $C$, and  $C$ is a finite group belonging to $\CC$, has $\mm$ different solutions.
\end{prop}
	\begin{proof}
	The first direction is trivial. For the second direction, we use induction on the length of a series of normal subgroups of $C$ that are contained in $K$. Denote it by $l_C(K)$. If $l_C(K)=1$ then by assumption, we are done. Assume that $l_C(K)=n>1$ and that the claim holds for all $m<n$. Since $l_C(K)>1$, there is a subgroup $1\ne A\unlhd C$ s.t $A\subset K$. We can look at the following embedding problem
	\[
	\xymatrix@R=14pt{ & & &G \ar@{->>}[d]^{\varphi}& \\
		1 \ar[r] & K/A\ar[r] & C/A \ar[r]&D \ar[r]&1\\
	}
	\]
	$l_{C/A}(K/A)<l_C(A)$, hence by induction assumption. there are $\mathbf{m}$ different solutions $\psi_i:G\to C/A$. Now, for each solution, we can build a new embedding problem
	\[
	\xymatrix@R=14pt{ & & &G \ar@{->>}[d]^{\psi_i}& \\
		1 \ar[r] & A\ar[r] & C \ar[r]&C/A \ar[r]&1\\
	}
	\]
	$l_C(A)<l_C(K)$, so again by induction assumption, it has $\mathbf{m}$ different solutions. Every such solution is a solution for the original embedding problem, so there are $\mathbf{m}\cdot \mathbf{m}=\mathbf{m}$ solutions.
	
\end{proof}
This refinement will be very useful in the proof of Theorem \ref{MainProp}.

Now we will present a way to determine whether a profinite completion of a given abstract group is free pro-$\CC$. 
	\begin{prop}\label{free_completion_of_abstract_group}
	Let $G$ be an abstract group, such that $\omega_0(\widehat{G})=\mm$. Then, $\widehat{G}$ is a free pro-$\CC$ group of rank $\mathbf{m}$ if and only if every \textit{abstract} embedding problem
	\[
	\xymatrix@R=14pt{ & & &G \ar@{->>}[d]^{\varphi}& \\
		1 \ar[r] & K\ar[r] & C \ar[r]^{\alpha}&D \ar[r]&1\\
	}
	\]
	with $K\ne 1$ minimal normal non trivial and $C$  a finite group belonging to $\CC$, has $\mathbf{m}$ different \textit{abstract} solutions.
\end{prop}
\begin{proof}
	For the first direction, Assume that $\widehat{G}$ is free pro-$\CC$. Let
	\[
	\xymatrix@R=14pt{ & & &G \ar@{->>}[d]^{\varphi}& \\
		1 \ar[r] & K\ar[r] & C \ar[r]^{\alpha}&D \ar[r]&1\\
	}
	\]
	be an abstract embedding problem as in the proposition. By definition of the profinite completion, $\varphi$ can be lifted to a continuous epimorphism $\hat{\varphi}:\widehat{G}\to D$, s.t $\hat{\varphi}=\varphi\circ i$, when $i:G\to \widehat{G}$ is the natural homomorphism. Look at the new embedding problem
	\[
	\xymatrix@R=14pt{ & & &\widehat{G} \ar[d]^{\hat
			\varphi}& \\
		1 \ar[r] & K\ar[r] & C \ar[r]^{\alpha}&D \ar[r]&1\\
	}
	\]
	Since $\widehat{G}$ is free pro-$\CC$ group of rank $\mm$, there are $\mathbf{m}$ different continuous epimorphisms $\psi_i:\widehat{G}\to C$ such that $\hat{\varphi}=\alpha\circ \psi_i$. Now look at the maps of the form $\psi_i\circ i:G\to C$. Since $i(G)$ is dense in $\widehat{G}$ and $\psi_i$ is continuous for all $i$, we get that $(\psi_i\circ i)(G)$ is dense in $C$. But $C$ is finite, and thus discrete, so $(\psi_i\circ i)(G)=C$. I.e, $\psi_i\circ i$ is onto. In addition, for all $i$, $\varphi=\hat{\varphi}\circ i=\alpha\circ\psi_i\circ i$.
	Finally, for all $i\ne j$ we have $\psi_i\circ i\ne \psi_j\circ i$, since $i(G)$ is dense in $\widehat{G}$, and hence if two continuous functions to an Hausdorff identify on it, then they are equal.
	
	For the second direction, look at the following diagram:
	\[
	\xymatrix@R=14pt{ & & &\widehat{G} \ar@{->>}[d]^{\varphi}& \\
		1 \ar[r] & K\ar[r] & C \ar[r]^{\alpha}&D \ar[r]&1\\
	}
	\]
	where $\varphi$ is continuous, and $C\in \CC$ is finite. In order to prove that $\widehat{G}$ is free of rank $\mathbf{m}$, we need to find $\mathbf{m}$ different solutions. Let us look at the following diagram
	\[
	\xymatrix@R=14pt{ & & &G \ar[d]^{\varphi\circ i}& \\
		1 \ar[r] & K\ar[r] & C \ar[r]^{\alpha}&D \ar[r]&1\\
	}
	\]
	$\varphi\circ i$ is onto, by the same argument as above. By assumption, $G$ admits $\mathbf{m}$ different epimorphisms $\psi_i:G\to C$ s.t $\alpha\circ \psi_i=\varphi\circ i$. For every such solution, define its lifting $\hat{\psi_i}:\widehat{G}\to C$, which is a continuous epimorphism. All these liftings are different, and they satisfy $\alpha\circ \hat{\psi}=\varphi$ since these are two continuous functions to an Hausdorff space, identifying on the dense subset $i(G)$.
\end{proof}
Following this proposition, we conclude that:
\begin{prop}
	Let $F(\mathbf{m})$ be a free pro-$\CC$ group of rank $\mathbf{m}$. Then, $\widehat{F_{\CC}(\mathbf{m})}$, the profinite completion of $F_{\CC}(\mathbf{m})$, is a free pro-$\CC$ group if and only if every \textit{abstract} embedding problem
	\[
	\xymatrix@R=14pt{ & & &G \ar@{->>}[d]^{\varphi}& \\
		1 \ar[r] & K\ar[r] & C \ar[r]^{\alpha}&D \ar[r]&1\\
	}
	\]
	with $K$ non trivial minimal normal subgroup, $C\in \CC$ finite, and $\varphi$ is not necessarily continuous, has $2^{2^{\mathbf{m}}}$ different \textit{abstract} solutions.
\end{prop}
\begin{proof}
	By Proposition \ref{free_completion_of_abstract_group} it is enough to prove that $\omega_0(\widehat{F_{\CC}(\mathbf{m})})=2^{2^{\mm}}$. By definition, the rank equals the cardinality of a set of generators converging to 1. By \cite[Proposition 2.6.2]{ribes2000profinite}, this is equal to $\omega_0(\widehat{F_{\CC}(\mathbf{m})})$, the number of normal open subgroups. But by \cite[Proposition 3.2.2]{ribes2000profinite} it equals to the number of normal subgroups of finite index in $F(\mathbf{m})$. or the same reasons, $\omega_0(F_{\CC}(\mathbf{m}))=\mathbf{m}$. Thus, $F(\mathbf{m})$ is a subgroup of the multiple of $\mathbf{m}$ finite groups, so $|F(\mathbf{m})|\leq 2^{\mathbf{m}}$, so the number of finite abstract quotients is less or equal to $2^{2^{\mathbf{m}}}$. On the hand, by the proof of Lemma \ref{basic lemma} $F_{\CC}(\mm)$ admits at least $2^{2^{\mathbf{m}}}$ finite quotients. So, $rank(\widehat{F(\mathbf{m})})=\omega_0(\widehat{F(\mathbf{m})})=2^{2^{\mathbf{m}}}.$
\end{proof}
At the end of this section, we will prove the abstract case:
\begin{lem}\label{abstract case}
	Let $F_{\operatorname{abs}}(\mm)$ denotes the free abstract group on $\mm$ generators. Then, $\widehat{F_{\operatorname{abs}}(\mm)}\cong F(2^{\mm})$.
\end{lem}
\begin{proof}
	First notice that $\omega_0(\widehat{F_{\operatorname{abs}}(\mm)})$ equals the number of open subgroups of $\omega_0(\widehat{F_{\operatorname{abs}}(\mm)})$ which in turn equals the number of finite index subgroup of $ F_{\operatorname{abs}}(\mm)$. Every finite index subgroup of $F_{\operatorname{abs}}(\mm)$ corresponds to some projection $F_{\operatorname{abs}}(\mm)\to A$ for a finite group $A$, and every such projection corresponds to some function $X\to A$ where $X$ is a basis for $F_{\operatorname{abs}}(\mm)$. There are $2^{\mm}$ different functions $X\to A$ and since there are countably many finite groups, there are $\aleph_0\cdot2^{\mm}=2^{\mm}$ finite index subgroups.
	Let 
		\[
	\xymatrix@R=14pt{ & & &F_{\operatorname{abs}}(\mm)\ar@{->>}[d]^{\varphi}& \\
		1 \ar[r] & K\ar[r] & C \ar[r]^{\alpha}&D \ar[r]&1\\
	}
	\]
	be a finite abstract embedding problem for $F_{\operatorname{abs}}(\mm)$. Let $X$ be some basis for $F_{\operatorname{abs}}(\mm)$. Define a function $\psi:X\to C$ such that for every $x\in X,\psi(x)\in \alpha^{-1}(\varphi(x))$. Since $|X|=\mm$ and $D$ is finite, there is some $d\in D$ such that $|\varphi^{-1}(d)|=\mm$. Use $\varphi^{-1}(d)$ to cover all $\alpha^{-1}(d)$. That can be done in $2^{\mm}$ different ways. Every such function can be extended to an epimorphism $\bar{\psi}:F_{\operatorname{abs}}(\mm)\to C$ which makes the diagram
		\[
	\xymatrix@R=14pt{ & & &F_{\operatorname{abs}}(\mm)\ar@{->>}[d]^{\varphi} \ar@{->>}[ld]_{\psi}& \\
		1 \ar[r] & K\ar[r] & C \ar[r]^{\alpha}&D \ar[r]&1\\
	}
	\]
	commutative.
\end{proof}
	\section*{Special varieties}
	\subsection*{The pro-$p$ case}
	In this subsection we deal with the case of $\CC$ being the variety of all finite $p$-groups. Free pro-$p$ groups have a special characterization:
	\begin{defn}
		Let $G$ be a pro-$\CC$ group. We say that $G$ is \textit{projective} in the category of pro-$\CC$ groups, if every embedding problem of the form 
			\[
		\xymatrix@R=14pt{  &G\ar@{->>}[d]^{\varphi} & \\
			 C \ar[r]^{\alpha}&D \ar[r]&1\\
		}
		\]
		where $C,D$ are pro-$\CC$ groups, is weakly solvable. I.e, there is a homomorphism $\psi:G\to C$ which makes the diagram 
		\[
		\xymatrix@R=14pt{  &G\ar@{->>}[d]^{\varphi} \ar[ld]_{\psi} & \\
		C \ar[r]^{\alpha}&D \ar[r]&1\\
	}
\]
commutative.
	\end{defn}
\begin{lem}\cite[Lemma 7.6.1]{ribes2000profinite}
	It is enough to consider $C$ and $D$ to be finite.
\end{lem}
\begin{lem}\cite[Lemma 7.6.3]{ribes2000profinite}
	Let $G$ be a pro-$\CC$ group. Then $G$ is projective if and only if $G$ is isomorphic to a closed subgroup of a free pro-$\CC$ group.
\end{lem} 
\begin{cor}\label{closed subgroup of projective}
	A closed subgroup of a pro-$\CC$ projective group is projective. 
\end{cor}
\begin{thm}\cite[Theorem 7.7.4]{ribes2000profinite}
	Let $G$ be a pro-$p$ group. Then $G$ is a free pro-$p$ group if and only if $G$ is projective.
\end{thm}
In \cite{bar-on_2021} the author proved that the profinite completion of a profinite projective group is projective.
\begin{cor}
	Let $G$ be a free pro-$p$ group. Then $\hat{G}$ is a free pro-$p$ group too.
\end{cor}
\begin{proof}
	By the main Theorem of \cite{bar-on_2021}, $\hat{G}$ is projective. So the only thing left to show is that $\hat{G}$ is a pro-$p$ group. By the explicit construction of $\hat{G}$, its continuous finite quotients are exactly the finite abstract quotients of $G$. By \cite[Proposition 4.2.3]{ribes2000profinite}, the order of every finite (abstract) quotient of a profinite group $G$ divides the order of $G$. So we are done.
\end{proof}
Now we want to deal with the completions of higher order. Our main theorem will be the following:
\begin{thm}\label{free pro p}
	Let $G$ be a free pro-$p$ group of infinite rank. Then, for every ordinal $\alpha$, $G_{\alpha}$ is a free pro-$p$ group.
\end{thm}
We will prove it in two steps:
\begin{lem}\label{orders in the chain}
	Let $G$ be a nonstrongly complete profinite group. Then for every ordinal $\alpha$, $o(G_{\alpha})=o(G)$.
\end{lem}
\begin{proof}
	Recall that the order of a profinite group equals the $\operatorname{lcm}$ of the orders of its finite continuous quotients. We will prove the claim by transfinite induction:
	
	$\alpha=0$ is obvious.
	
	$\alpha=\beta+1$: By induction assumption, $o(G_{\beta})=o(G)$. We will prove that for every profinite group $H$, $o(H)=o(\hat{H})$. Indeed, $o(\hat{H})$ equals the $\operatorname{lcm}$ of the orders of all its finite continuous quotients, which are precisely the finite abstract quotients of $H$. Since this set contains all the finite continuous quotients of $H$, then $o(\hat{H})\geq o(H)$. On the other hand, by \cite[Proposition 4.2.3]{ribes2000profinite} the orders of the abstract finite quotients of $H$ divides $o(H)$. So, $o(\hat{H)}\leq o(H)$ and we are done.  
	
	$\alpha$ is limit: The order of $G_{\alpha}$ is the $\operatorname{lcm}$ of the orders of its finite continuous quotients which are precisely the finite quotients of $H_{\alpha}$. On one hand, being the direct limit of $\{G_{\beta}\}_{\beta<\alpha}$ every finite quotient of $H_{\alpha}$ is in fact a finite (abstract) quotient of some $G_{\beta}$ for $\beta<\alpha$. So by \cite[Proposition 4.2.3]{ribes2000profinite} and induction assumption, its order divides $o(G_\beta)=o(G)$. On the other hand, by \cite[Proposition 5]{BarOn2018tower}, every finite quotient of $G$ cab by lifted to a finite quotient of $H_{\alpha}$, so $o(G_{\alpha})=\operatorname{lcm}\{H_{\alpha}/U\}_{U\unlhd_f H_{\alpha}}=o(G)$.
\end{proof}
\begin{prop}\label{projective in the chain}
	Let $G$ be a nonstrongly complete pro-$\CC$ projective group. Then for every ordinal $\alpha$, $G_{\alpha}$ is projective. 
\end{prop}
\begin{proof}
	This proposition was proved in \cite{cohomologicalproperties}. We give here another proof. Let $\alpha$ be some ordinal and assume that $G_{\beta}$ is projective for every $\beta<\alpha$.
	
	$\alpha=0$: this is the assumption of the proposition.
	
	$\alpha=\beta+1$: By induction assumption $G_{\beta}$ is profinite projective. So this is exactly the Main Theorem of \cite{bar-on_2021}.
	
	$\alpha$ is limit: By \cite[Lemma 2]{bar-on_2021} in order to prove that $G_{\alpha}$ is projective, we need to prove that every finite (abstract) embedding problem for $H_{\alpha}$ is weakly solvable. Let   
		\[
	\xymatrix@R=14pt{  &H_{\alpha}\ar@{->>}[d]^{\varphi}  & \\
		C \ar[r]^{f}&D \ar[r]&1\\
	}
	\]
	be such an embedding problem. For all $\beta<\alpha$ we look at the reduced homomorphism $\varphi|_{G_{\beta}}:G_{\beta}\to D$. Since $D$ is finite, there exists some $\gamma$ such that for all $\gamma\leq \beta < \alpha$ $\varphi|_{G_{\beta}}$ is an (abstract) epimorphism.
	by assumption each $G_{\beta}$ is projective, so by \cite{bar-on_2021}, the set $WS(\beta)$ of abstract weak solutions for $G_{\beta}$ is nonempty. Notice that $WS(\beta)\subset C^{G_{\beta}}$, the set of all functions from $G_{\beta}$ to $C$, which is a profinite space with the product topology. Moreover, being a homomorphism means that for all $g,h\in G_{\beta}$, $\psi(g)\psi(h)=\psi(gh)$, thus, the set of all homomorphisms is precisely $$\bigcap_{g,h\in G_{\beta}}(\bigcup_{a,b\in C}\{a\}_g\times \{b\}_h\times \{ab\}_{gh}\times \prod_{g,h,gh\ne x\in G_{\beta}}C)$$ which is a closed subset. In addition, being a solution is equal to be contained in $\prod_{g\in G_{\beta}}f^{-1}[\{\varphi(g)\}]$ which is a closed subset too. Thus, $WS(\beta)$ is a nonempty profinite set. Eventually, the reduction function $WS(\beta)\to WS(\delta)$ when $\gamma\leq \delta<\beta$ is nothing but the projection map $C^{G_{\beta}}\to C^{G_{\delta}}$, considered $G_{\delta}$ as a subgroup of $G_{\beta}$, and thus continuous. So, by \cite[Proposition 1.1.4]{ribes2000profinite} the inverse limit ${\invlim}_{\gamma\leq \alpha < \beta}WS(\alpha)$ is nonempty. Eventually, an element  in the inverse limit is precisely a weak solution for $H_{\alpha}$.
\end{proof}
\begin{proof}[Proof of Theorem \ref{free pro p}]
Let $G$ be a free pro-$p$ group of infinite rank. Then by Lemma \ref{basic lemma} $G$ is not strongly complete. By Lemma \ref{orders in the chain}, for every ordinal $\alpha$, $G_{\alpha}$ is a pro-$p$ group. Thus, $G_{\alpha}$ is a free pro-$p$ if and only if it is projective. So by Proposition \ref{projective in the chain} we are done. 
\end{proof}
	\subsection*{The pronilpotent case}
	A pronilpotent group is an inverse limit of finite nilpotent groups. By \cite[Proposition 2.3.8]{ribes2000profinite} every pronilpotent group is a direct product of pro-$p$ groups. 
		\begin{lem}\label{Lemma pronilpotent maps}
		Let $G$ be a pronilpotent group, i.e $G=\prod _pG_p$ where $G_p$ denotes a $p$ - group. Then every abstract epimorphism $\varphi:G\to H$ where $H$ is a finite $q$ - group for some prime $q$, projects through the natural projection on $G_q$. 
	\end{lem}
	\begin{proof}
		Look at $\varphi'=\varphi|_{\prod_{p\ne q}G_p}:\prod_{p\ne q}G_p\to H$. Assume by contradiction that the homomorphism is not trivial. Then, $Im(\varphi')$ is a $q$ - subgroup. But, by \cite[Proposition 4.2.1]{ribes2000profinite} , $|H|\mid |\prod_{p\ne q}G_p|$, a contradiction.
	\end{proof}
The previous Lemma has the following application:
	\begin{prop}\label{completiom_of_pronilpotent}
		Let $G=\prod_{p}G_p$ be a pronilpotent group. Then, $\widehat{G}\cong \prod_p\widehat{G_p}$. 
	\end{prop}
	\begin{proof}
		Let $\varphi:G\to H$ be an abstract epimorphism onto a finite group. By \cite[Corollary 4.2.4]{ribes2000profinite}, $H$ is nilpotent, Denote $H=H_{p_1}\times\cdots\times H_{p_n}$ where $H_{p_i}$ is a finite $p_i$-group. Then $\varphi=(\varphi_{p_1},...,\varphi_{p_n})$, where $\varphi_{p_i}:G\to H_{p_i}$. By \ref{Lemma pronilpotent maps}, each $\varphi_{p_i}$ projects through $G_{p_i}$. So, $$\widehat{G}={\invlim}_{K\unlhd_f G}G/K\cong {\invlim} G_{p_{i_1}}/K_{p_{i_1}}\times\cdots G_{p_{i_m}}/K_{p_{i_m}}$$ 
	\end{proof}
	\begin{cor}
		Let $G=\prod G_p$ be a pronilpotent group. Then $G$ is strongly complete iff for every $p$, $G_p$ is strongly complete.
	\end{cor}
	\begin{prop}\label{sylow subgroups of free pronilpotent groups}
		Let $G=\prod_pG_p$ be a pronilpotent group. Then, $G$ is a free pronilpotent group of rank $\mm$ iff for every $p$, $G_p$ is a free pro - $p$ group of rank $\mm$.
	\end{prop}
	\begin{proof}
		First, assume that $G$ is free pronilpotent of rank $\mm$. Notice that $\omega_o(G_p)\leq \omega_0(G)$. Look at the following embedding problem:
		\[
		\xymatrix@R=14pt{ & & &G_p \ar@{->>}[d]^{\varphi}& \\
			1 \ar[r] & K\ar[r] & C \ar[r]^{\alpha}&D \ar[r]&1\\
		}
		\]
		where $C$ is a finite $p$- group. We can lift $\alpha$ to an epimorphism from $G$ by the natural projection. Then it has $\mm$ solutions, which project through $G_p$. So, the problem has $\mm$ solutions. Hence, $G_p$ is a free pro- $p$ group of rank $\mm$.
		Now, assume that for every $p$, $G_p$ is a free pro- $p$ group of rank $\mm$. Notice that $\omega_0(G)=\mm$. Look at the following embedding problem
		\[
		\xymatrix@R=14pt{ & & &G \ar@{->>}[d]^{\varphi}& \\
			1 \ar[r] & K\ar[r] & C \ar[r]^{\alpha}&D \ar[r]&1\\
		}
		\]
		where $C$ is a nilpotent group. Denote $C\cong C_{p_1}\times \cdots C_{p_n}$. We can split the diagram into $n$ embedding problems of $p$ groups. By assumption, each one of the problems has $\mm$ solutions, and thus the given embedding problem has $\mm$ solutions. Hence, $G$ is a free pronilpotent group of rank $\mm$ 
	\end{proof}
	\begin{cor}\label{pronilpotent the successor case}
		By Theorem \ref{free pro p}, Proposition \ref{completiom_of_pronilpotent} and Proposition \ref{sylow subgroups of free pronilpotent groups} we conclude that the profinite completion of a free pronilpotent group is a free pronilpotent group.
	\end{cor}
\begin{thm}
	Let $G$ be a free pronilpotent group of infinite rank. Then for every ordinal $\alpha$, $G_{\alpha}$ is free pronilpotent.
\end{thm}
\begin{proof}
First, by Lemma \ref{basic lemma} $G$ is not strongly complete. We will prove the theorem by transfinite induction.

$\alpha=0$: this is the assumption.

$\alpha=\beta+1$: By induction assumption, $G_{\beta}$ is free pronilpotent, so this is exactly Corollary \ref{pronilpotent the successor case}.

$\alpha$ is limit: First, we would like to prove that $G_{\alpha}$ is pronilpotent. $G_{\alpha}$ is defined to be the inverse limit of all the finite quotients of $H_{\alpha}$. Thus, we need to show that every finite quotient of $H_{\alpha}$ is nilpotent. But since $H_{\alpha}={\dirlim}_{\beta<\alpha}G_{\beta}$, every finite quotient of $H_{\alpha}$ is also a finite quotient of $G_{\gamma}$ for some $\gamma<\alpha$. By induction assumption $G_{\gamma}$ is pronilpotent, so by \cite[Corollary 4.2.4]{ribes2000profinite} every abstract finite quotient of $G_{\gamma}$ is nilpotent, and we are done. Now, one sees immediately that a free pro-$\CC$ group is projective. So, by Proposition \ref{projective in the chain} $G_{\alpha}$ is projective. By Corollary \ref{closed subgroup of projective}, a closed subgroup of a projective group is projective, so all the $p$-Sylow subgroups of $G_{\alpha}$ are projective pro-$p$ groups- i.e, free pro-$p$ groups. By \ref{sylow subgroups of free pronilpotent groups}, the only thing left to show is that $\operatorname{rank}((G_{\alpha})_p)=\omega_0(G_{\alpha})$. Since $(G_{\alpha})_p$ is a closed subgroup of $G_{\alpha}$, by \cite[Corollary 2.6.5]{ribes2000profinite} $ \operatorname{rank}((G_{\alpha})_p)=\omega_0((G_{\alpha})_p)\leq \omega_0(G_{\alpha})$. For the second direction, recall that by \cite[Corollary 4]{BarOn+2021}, if $H$ is a nonstrongly complete profinite group, then the local weight is determined only by $\omega_0(H)$. Since $\omega_0((G)_p)=\omega_0(G)$, we get that $ \omega_0((G)_p)_{\alpha})=\omega_0(G_{\alpha})$. By \cite[Lemma 1.2.5]{fried2006field} a profinite group $H$ induces the full profinite topology on every closed subgroup, meaning that every open subgroup of the closed subgroup contains the intersection of the closed subgroup with some open subgroup of $H$. By \cite[Lemma 3.2.6]{ribes2000profinite} it implies that the profinite completion preserves inclusion of closed subgroups. Then, by an easy application of transfinite induction, one gets that if $K\leq_cH$ then for every ordinal $\alpha$, $K_{\alpha}\leq_cH_{\alpha}$. In our case, $(G_p)_{\alpha}\leq_cG_{\alpha}$. Since $(G_p)_{\alpha}$ is pro-$p$ group (Lemma \ref{orders in the chain}), and $G_{\alpha}$ is a product of pro-$p$ groups, $(G_p)_{\alpha}\leq_c(G_{\alpha})_p$. Hence, $\omega_o(G_{\alpha})=\omega_0((G_p)_{\alpha})\leq\omega_0((G_{\alpha})_p)$. We are done. 
\end{proof}
	\subsection*{The pro-abelian case}
	\begin{prop}
		The profinite completion of a free abelian profinite group is free abelian profinite.
	\end{prop}
	\begin{proof}
		By \cite[Example 3.3.9 (c)]{ribes2000profinite} the free abelian profinite group of rank $\mm$ is the product of $\mm$ copies of $\mathbb{Z}$, and its $p$-Sylow groups are free pro-$p$ abelian groups. By Proposition \ref{completiom_of_pronilpotent} it's enough to prove the claim for free pro-$p$ abelian groups. Denote by $F_p=\prod_{\mm}\mathbb{Z}_p$ the free pro-$p$ abelian group, and let 
		\[
		\xymatrix@R=14pt{ & & &F_p \ar@{->>}[d]^{\varphi}& \\
			1 \ar[r] & K\ar[r] & C \ar[r]^{\alpha}&D \ar[r]&1\\
		}
		\]
		be an abstract embedding problem where $K$ is minimal normal. Then, $K\cong C_p$, and there are two options for $C$ and $D$.
		\begin{itemize}
			\item $C\cong K\times D$ with the natural projection.
			\item $D\cong C_{p^n}\times E$ and $C\cong C_{p^{n+1}}\times E$ with the natural projection. 
		\end{itemize}
		In case 1 we simply define $\psi=(\eta,\varphi)$ where $\eta$ is some abstract epimorphism from $F_p$ to $K$. Since the free pro-$p$ abelian group of rank $\mm$ projects on $\prod_{\mm}C_p$, there are $2^{2^{\mm}}$ such epimorphisms.
		In case 2, $\varphi=(\varphi_1,\varphi_2)$ where $\varphi_1:F_p\to C_{p^n}$ and $\varphi_2:F_p\to E$. Define $\psi=(\eta,\varphi_2)$ where $\eta:F_p\to C_{p^{n+1}}$ is an abstract solution to the reduced diagram
		\[
		\xymatrix@R=14pt{ & & &F_p \ar@{->>}[d]^{\varphi_1}& \\
			1 \ar[r] & K\ar[r] & C_{p^{n+1}} \ar[r]^{\alpha}&C_{p^n} \ar[r]&1\\
		}
		\]
		We need to show that there are $2^{2^{\mm}}$ options to such $\eta$.
		
		Notice that any abstract epimorphism from $\prod_{\mm}\mathbb{Z}_p$ to $C_{p^k}$ splits through $(\prod_{\mm}\mathbb{Z}_p)/p^k\cong \prod_{\mm}C_{p^k}$. Thus we have the following diagram:
		
		\[
		\xymatrix@R=14pt{  &F_p \ar@{->>}[d] \ar@{->>}[dl]  \\
			\prod_{\mm}C_{p^{n+1}} \ar@{->>}[r] & \prod_{\mm}C_{p^n} \ar@{->>}[d] \\
			C_{p^{n+1}} \ar@{->>}^{\alpha}[r]&C_{p^n} \\
		}
		\]
		Recall that $\prod_{\mm}C_{p^k}$ is a free module over $\mathbb{Z}/p^k\mathbb{Z}$ and so it is isomorphic to $\bigoplus_{2^{\mm}}\mathbb{Z}/p^k\mathbb{Z}$. So by defining a function on the basis which sends every element to a possible origin in $C_{p^{n+1}}$ of its image in $C_{p^n}$, in such a way that will cover all the origins of some element, we get a weak solution which covers the kernel, which is thus a solution. There are $2^{2^{\mm}}$ ways to do so. 
	\end{proof}
\begin{thm}
	Let $G$ be a free profinite abelian group. Then for every $\alpha$, $G_{\alpha}$ is free profinite abelian of the corresponding rank.
\end{thm}
\begin{proof}
	By Proposition \ref{projective in the chain}, $G_{\alpha}$ is projective for every $\alpha$, so by Lemma \ref{closed subgroup of projective} it is isomorphic to a closed subgroup of a free profinite abelian group, and thus has the form $\prod_{p}(\prod_{I_{p,\alpha}}\mathbb{Z}_{p})$ where $p$ runs over the set of all primes, and $I_{\alpha,p}$ is some indexing set. We only left to show that $|I_{\alpha,p}|=\omega_0(G_{\alpha})$. We prove it by transfinite induction. Let $\beta$ be some ordinal and assume that the claim holds for all $\alpha<\beta$.
	
	By Proposition \ref{completiom_of_pronilpotent} $$\widehat{\prod_{p}(\prod_{I_{p,\alpha}}\mathbb{Z}_{p})}\cong \prod_{p}(\widehat{\prod_{I_{p,\alpha}}\mathbb{Z}_{p}})$$
	
	Thus, by \cite[Theorem 2]{BarOn+2021} if $\alpha=\beta+1$, $|I_{p,\alpha}|=2^{2^{|I_{p,\alpha}|}}=2^{2^{\omega_0(G_{\alpha})}}=\omega_0(G_{\beta}).$
	
	If $\alpha$ is limit we want to show that for every $p$, the $p$-part of $G_{\alpha}$ has $\omega_0(G_{\beta})$ projections on $C_p$. Indeed, since $G_0$ is free, it projects onto $\prod_{\omega_0(G_0)}C_p$, so we can apply the proof of Proposition 10 in \cite{bar-on_2021} with $S=C_p$.   
\end{proof}

	\section*{Partial results for the general case}
	We give some results which hold in general for every variety $\CC$, and in particular for the variety of all finite groups. 
		\begin{thm}\label{MainProp}
		Let $G$ be a free pro-$\CC$ group of rank $\mathbf{m}$. Then every embedding problem
		\[
		\xymatrix@R=14pt{ & & &G \ar@{->>}[d]^{\varphi}& \\
			1 \ar[r] & K\ar[r] & C \ar[r]^{\alpha}&D \ar[r]&1\\
		}
		\]
		with $C\ne 1$ is finite in $\CC$ and $\varphi$ is continuous, has $2^{2^{\mathbf{m}}}$ \textit{abstract} solutions.
	\end{thm}
	Before proving it, we need the following lemma:
	\begin{lem}\label{homomorphism to the product is onto}
		Let $G$ be a profinite group, $U\unlhd_oG$ and $\mathcal{M}$ a set of normal subgroups $H\unlhd_oG$, s.t, for all $H\in \mathcal{M}$, $H\subseteq U$, and $U/H$ is a nontrivial minimal normal subgroup in $G/H$. Denote $M=\bigcap_{H\in \mathcal{M}}H$. Then, there is a subset $\mathcal{N}\subseteq \mathcal{M}$ such that the natural homomorphism $$\varphi_{\mathcal{N}}:U\to \prod_{H\in \mathcal{N}}U/H$$ is an epimorphism, and $M=\ker(\varphi_{\mathcal{N}}).$
	\end{lem}
	\begin{proof}
		The proof is similar to this of \cite{ribes2000profinite} Lemma 8.2.2.
		
		Let $\Omega$ denote the family of subsets $\mathcal{L}\subseteq \mathcal{M}$
		for which $$\varphi_{\mathcal{L}}:U\to \prod_{H\in \mathcal{L}}U/H$$ is an epimorphism. $\Omega$ is not empty, since it contains all the singletons.	
		Since $\prod_{H\in \mathcal{L}}U/H$ is an inverse limit of direct products over finite sets, one deduces that $\mathcal{L}\in \Omega$ if and only if $\varphi_{\mathcal{F}}$ is an epimorphism for each of its finite subsets $\mathcal{F}$, i.e, $\mathcal{F}\in \Omega$ for each of its finite subsets $\mathcal{F}$. Therefore, $\Omega$, ordered by inclusion, is an inductive set.
		Hence, by Zorn's lemma, there exists a maximal subset $\mathcal{N}$ in $\Omega$. To finish the proof it suffices to show that $M=\ker{\varphi(\mathcal{N})}$. Obviously, $\ker(\varphi_{\mathcal{N}})=\bigcap_{H\in \mathcal{N}}H$. Denote this subgroup by $N$. We need to show that $N=M$. It is obvious that $M\leq N$. If $M<N$, then there would exist some $K\in \mathcal{M}$ with $K\cap N<N$.
		
		Notice that since $K\in \mathcal{M}$, there is no normal subgroup $A\unlhd_oG$ that satisfies $K<A<U$. Since $N$ is the intersection of all normal subgroups of $G$ that are contained in $U$, it is also a normal subgroup of $G$ that is contained in $U$. Thus, $K\leq KN\leq U$ is a normal subgroup of $G$. If $KN=K$ then $K\cap N=N$, a contradiction. So, $KN=U$. Thus, $U/(K\cap N)\cong U/K\times U/N$. This implies that $\mathcal{N}\cup \{K\}\in \Omega$, contrary to the maximality of $\mathcal{N}$. Thus, $N=M$
	\end{proof}
	\begin{proof}[Proof of Theorem \ref{MainProp}]
		By Proposition \ref{minimal normal}, we may assume that $K$ is a nontrivial minimal normal subgroup of $C$. Let $U=\ker(\varphi)$. Since $D$ is finite and $\varphi$ is continuous, $U$ is an open normal subgroup of $G$. Since $G$ is a free pro-$\CC$ group, this embedding problem has $\mathbf{m}$ different (continuous) solutions, $\psi_i:G\to C$. Denote $H_i=\ker(\psi_i)$, and let $\mathcal{M}=\{H_i\}$. Obviously, for all $i$, $H_i\unlhd_oG$, and $H_i\subseteq U$. Moreover, by the assumption that $K$ is a normal minimal subgroup of $C$, we get that $U/H_i$ is a minimal normal subgroup of $G/H_i$. $\mathcal{M}$ satisfies the conditions of Lemma \ref{homomorphism to the product is onto}, thus there is a subset $\mathcal{N}\subseteq \mathcal{M}$ s.t $U/M\cong \prod_{H_i\in \mathcal{N}}U/H_i$ where $M=\bigcap_{H_i\in \mathcal{M}}H_i$. By computation of local weight we get that $|\mathcal{N}|=\mathbf{m}$. Explanation: On one hand, $\omega_0(\prod_I G_i)$ where all $G_i$ are finite group, equals $|I|$. On the other hand, $\omega_0(\prod_{H_i\in \mathcal{N}}U/H_i)\leq \omega_0(U)=\omega_0(G)=\mm$. In addition, for every $H_i\in \mathcal{M}$ $H_i/{\bigcap \mathcal{M}}$ is a different open subgroup of $\prod_{H_i\in \mathcal{N}}U/H_i$. So, $\omega_0(\prod_{H_i\in \mathcal{N}}U/H_i)\geq |\mathcal{M}|=\mm$. Thus, $\omega_0(\prod_{H_i\in \mathcal{N}}U/H_i)=\mm$.
		
		Let $\mathcal{F}$ be an ultrafilter on $\mathcal{N}$, and denote by $U_{\mathcal{F}}$ the subgroup it defines. It is known that $(\prod_{H_i\in \mathcal{N}}U/H_i)/U_{\mathcal{F}}\cong K$. Since $\varphi_{\mathcal{N}}$ is onto we get that $H=\varphi_{\mathcal{N}}^{-1}(U_{\mathcal{F}})$ is a normal subgroup of $U$, such that $U/H\cong K$.
		
		\begin{claim}
			$H\unlhd G$
		\end{claim}
		\begin{proof}
			Look at the following map: $$\varphi_{\mathcal{N}}:G\to \prod_{H_i\in \mathcal{N}}G/H_i$$ which is not necessarily onto. Denote by $G_{\mathcal{F}}$ the subgroup that the filter $\mathcal{F}$ creates in $\prod_{H_i\in \mathcal{N}}G/H_i$. Since the map from $U$ to $\prod_{H_i\in \mathcal{N}}U/H_i$ is onto, we get that $U_{\mathcal{F}}=\varphi_{\mathcal{N}}(U)\cap G_{\mathcal{F}}.$ In addition $U=\varphi_{\mathcal{N}}^{-1}(\prod_{H_i\in \mathcal{N}}G/H_i) $. Hence, $$H={\varphi_{\mathcal{N}}}^{-1}(U_{\mathcal{F}})={\varphi_{\mathcal{N}}}^{-1}(G_{\mathcal{F}}\cap \varphi_{\mathcal{N}}(U))={\varphi_{\mathcal{N}}}^{-1}(G_{\mathcal{F}})\cap U$$ the intersection of two normal subgroups.
		\end{proof}
		\begin{claim}
			$G/H\cong C$.
		\end{claim}
		\begin{proof}
			First recall that $(\prod_{H_i\in \mathcal{N}}U/H_i)/U_{\mathcal{_F}}\cong K$, and $(\prod_{H_i\in \mathcal{N}}G/H_i)/G_{\mathcal{F}}\cong C$. Since the map from $U$ is onto we get that $U/H\cong K$, so $|G/H|=|C|.$
			Assume that $|K|=t$ and $|D|=s$. Since the map from $U$ to the product is onto, there are elements $\{u_1,...,u_t\}$ which are representatives for $U/H_i$ for all $H_i\in \mathcal{N}$. Take a set of representatives $g_1,...,g_s$ of $U$ in $G$. Then $\{g_iu_j\}$ is a set of representatives of $H_i$ in $G$, for all $H_i\in \mathcal{N}$.
			In the isomorphism $(\prod_{H_i\in \mathcal{N}}G/H_i)/G_{\mathcal{F}}\cong C$, a set of representatives is the fixed tuples $\{(c)\}_c\in C$ when identifying $G/H_i$ with $C$ for all $H_i\in \mathcal{N}$. Let us look at the image of $g_1u_1$ in the product. It is equivalent to some fixed tuple, say $(c_1)$. It means that there is a subset $\mathcal{A}_1$ of $\mathcal{N}$ that lays in $\mathcal{F}$, such that for each subgroup in $H_i\in\mathcal{A}_1$ $\varphi_{H_i}(g_1u_1)=c_1$. Now look at the image of $(g_1u_2)$. For every $H_i\in \mathcal{N}$, $\varphi_{H_i}(g_1u_1)\ne \varphi_{H_i}(g_1u_2)$, so in the groups that in $\mathcal{A}_1$, $\varphi_{H_i}(g_1u_2)\ne c_1$. We can make a finite partition of $\mathcal{A}_1$ according to the image of $g_1u_2$. Since $\mathcal{F}$ is an ultrafilter, one of the subsets in the partition belongs to $\mathcal{F}$. Call it $\mathcal{A}_2$. So $\varphi_{\mathcal{N}}(g_1u_2)$ equivalent to the fixed tuple $(c_2)$, where $c_2\ne c_1$. We continue in this way, and since there are $ts=|K||D|=|C|$ representatives we get that the composed map $$G\to \prod_{H_i\in \mathcal{N}}G/H_i\to (\prod_{H_i\in \mathcal{N}}G/H_i)/G_{\mathcal{F}}\cong C$$ is onto. So, $G/{\varphi_{\mathcal{N}}}^{-1}(G_{\mathcal{F}})\cong C$. But, $H={\varphi_{\mathcal{N}}}^{-1}(G_{\mathcal{F}})\cap U$, so there is a natural epimorphism $G/H\to G/{\varphi_{\mathcal{N}}}^{-1}(G_{\mathcal{F}})\cong C$. Since $|G/H|=|C|$ we get that $G/H\cong C$.
		\end{proof}
		\begin{claim}
			The natural projection $\varphi_H:G\to G/H$ is a solution.
		\end{claim}
		\begin{proof}
			We left to show that $\varphi=\alpha\circ\varphi_H$. Well, let $g\in G$. By definition of $H$, there exist some $i$ s.t $\varphi_H(g)=\varphi_{H_i}(g)$. By definition of the $H_i$'s we know that $\varphi_{H_i}$ is a solution, so $\alpha(\varphi_{H_i}(g))=\varphi(g)$.
		\end{proof}
		Since there are $2^{2^{\mathbf{m}}}$ different ultrafilters on $\mathcal{N}$, and each one of them defines a different subgroup $H$, we get $2^{2^{\mathbf{m}}}$ different abstract solutions to the embedding problem.
	\end{proof}
	\begin{rem}
	Let $G$ be a group with $\omega_0(G)=\mathbf{m}$, and $H\unlhd_f G$ and suppose the diagram
	\[
	\xymatrix@R=14pt{ & & &G \ar@{->>}[d]^{\varphi}& \\
		1 \ar[r] & K\ar[r] & C \ar[r]^{\alpha}&G/H \ar[r]&1\\
	}
	\]
	where $C$ is finite and $K$ is a minimal normal subgroup, has $2^{2^{\mathbf{m}}}$ solutions $\beta_i:G\to C$. Then it has $2^{2^{\mathbf{m}}}$ solutions which satisfy $\overline{\ker \beta_i}=\overline{H}$.	
\end{rem}
\begin{proof}
	Denote these kernels by $T_i$. Obviously. $T_i\leq H$. Notice that since $T_i\unlhd G$, then $\overline{T_i}\unlhd G$. Moreover, since $T_i$ has a finite index, then so is $\overline{T_i}$. Thus, $\overline{T_i}$ is open. Obviously, $\overline{T_i}\leq \overline{H}$. Assume $\overline{T_i}\ne \overline{H}$. Then $H$ can't be contained in $\overline{T_i}$ So, $\overline{T_i}\cap H<H.$ But $T\leq \overline{T_i}\cap H$. Since $K$ is minimal normal subgroup, and $\overline{T_i}\cap H\unlhd G$, we get that $\overline{T_i}\cap H=T_i$. Hence, two solutions with the same closure which is not $\bar{H}$ must be equal. Since G admits only $\mathbf{m}$ different open normal subgroups, we get that there are $2^{2^{\mathbf{m}}}$ such $T_i$ for which $\bar{T_i}=\bar{H}$.
\end{proof}
\begin{rem}
	Let $F_{\CC}(\mathbf{m})$ be a free pro-$\CC$ group of rank $\mathbf{m}$. Assume that $\widehat{F_{\CC}(\mathbf{m})}$ is a free profinite group of rank $2^{2^{\mathbf{m}}}$, and $X$ is some basis for $\widehat{F_{\CC}(\mathbf{m})}$. Then, $X\cap F_{\CC}(\mathbf{m})$ is not a basis for $F_{\CC}(\mathbf{m})$.
\end{rem}
\begin{proof}
	Let $X_0$ be a basis for $F_{\CC}(\mathbf{m})$. Look at the group $G=\prod_{i=1}^{\infty}S$ where S is some finite group in $\CC$. Let $s\in S$ be a nontrivial element. Consider the set $A$, consisting of all the tuples $(a_i)$ of the form $$a_i=
	\begin{cases}
	e & i<j \\
	x& i=j\\
	s& i>j 
	\end{cases}$$ for some $j\in \mathbb{N}$ and $x\in S$.  One easily checks that this is a set of generators converging to 1 for $G$, So we can define an epimorphism $\varphi:F_{\CC}(\mathbf{m})\to G$ by sending a subset of cardinality $\aleph_0$ of $X_0$ bijectively to $A$, and the rest of the elements to 1. Let $\mathcal{F}$ be an ultrafilter on $I$, and let $H_\mathcal{F}$ be the subgroup it defines. Then $\varphi^{-1}(H_\mathcal{F})$ is a subgroup of finite index that doesn't contain an infinite subset of $X_0$. Its closure in $\widehat{F_{\CC}(\mathbf{m})}$ is an open subgroup that doesn't contain an infinite subset of $X_0$, since for every subgroup $K\leq _f F(\mathbf{m})$, $\overline{K}\cap F(\mathbf{m})=K$. So, $X_0$ can't be contained in any set converging to 1.
\end{proof}
	\section*{Abstract homomorphisms}
In this section we give two results regarding abstract homomorphisms from a free pro-$\CC$ group to itself.
	\begin{prop}
	Let $F$ be a free anabelian profinite group. I.e, a free pro-$\CC$ group of the class (which is not a variety) of all finite groups which have no abelian composition factors. Then any abstract automorphism of $F$ is continuous.
\end{prop}
\begin{proof}
	Let $\varphi:F\to F$ be an abstract automorphism. We need to show that for every $U\leq_oF$ $\varphi^{-1}(U)$ is open. It is enough to show this for every open normal subgroup. Indeed, if $U\leq_oG$, then $U$ contains some normal open subgroup $N$. Thus $\varphi^{-1}(N)\subseteq \varphi^{-1}(U)$. I.e, $\varphi^{-1}(U)$ contains an open subgroup and thus it is open.
	We will prove it by induction on $l(G/U)$, the length of a composition series for $G/U$.
	First case: $l(G/U)=1$ I.e, $G/U\cong A$ is simple. By assumption, it is simple nonabelian. By \cite[Lemma 8.2.2]{ribes2000profinite} $F/M\cong \prod A$ where $M$ is the intersection of all normal subgroups of $F$ with intersection isomorphic to $A$. $\varphi$ induces an automorphism $\bar{\varphi}:F/M\to F/M$. By \cite[Theorem 2.4]{kiehlmann2015abelian}  every such automorphism is continuous. Let $\pi:F\to F/M$ denotes the natural projection. Then, $\phi\circ\varphi=\bar{\varphi}\circ \pi$. So, $\phi\circ\varphi$ is continuous. Thus, $(\phi\circ\varphi)^-1(U/M)=\varphi^-1(U)$ is open.
	Now, assume $l(U)=n+1$. There exists some open normal subgroup $H$ such that $U/N$ is simple, and $l(H)=n$. By induction assumption, $\varphi^{-1}(H)$ is open. Every open subgroup of $F$ is isomorphic to $F$, so there is a continuous isomorphism $\psi:H\to \varphi^{-1}(H)$. $\varphi\circ \psi:H\to H$ is an automorphism, so again by induction assumption $(\varphi\circ \psi)^{-1}(U)$ is open. Since $\psi$ is a homeomorphism we get that $\varphi^{-1}(U)$ is open.    
\end{proof}
\begin{prop}\label{non continuou inner epimorphism}
	Let $F_{\CC}(\mm)$ be a free pro-$\CC$ group of infinite rank $\mm$. Then $F_{\CC}(\mm)$ admits a noncontinuous epimorphism $F_{\CC}(\mm)\to F_{\CC}(\mm)$.
\end{prop}
Before we can prove it, we need the following lemma:
	\begin{lem}\label{non continuous epimorphism from the product}
	Let $G\cong{\invlim}_I\{G_i,\varphi_{ij}\}$ be an infinite profinite group, expressed as an inverse limit of finite groups. Then there exists a noncontinuous epimorphism $\phi:\prod G_i\to G$.
\end{lem} 
\begin{proof}
	The proof is based on a similar argument shown in Proposition 2.1 \cite{klopsch2016abstract}. Let $\mathcal{F}$ be a non principal ultrafilter on $I$ which contains the set $[i,\infty), i\in I]\}$ where $[i,\infty)=\{j\in I:i\leq j\}$. Notice that since $I$ is an infinite directed poset with no maximal element, this set satisfies the finite intersection property i.e, every finite intersection contains a set of the form $[i,\infty)$, and in addition all these sets are infinite, so it can be extended to a non principal ultrafilter. Now for each $i\in I$ we define a projection $\phi_i;\prod G_i\to G_i$ by the following: Let $(g_i)\in prod G_i$. For each $x\in G_i$ we define $A_{x}$ to be the set of all $G_j$ such that $j\geq i$ and $\varphi_{ji}(g_j)=x$. This creates a finite partition of $[i,\infty)$ which belongs to $\mathcal{F}$, so there exists exactly one $x\in G_i$ such that $A_x\in \mathcal{F}$. So we define $\phi_i(g_i)=x$.
	
	\begin{claim}
		For every $i$, $\phi_i$ is an homomorphism.
	\end{claim}
	\begin{proof}
		Assume $\phi_i(g_i)=x$ and $\phi_i(g'_i)=y$. Then the set of all indexes $j$ such $j\geq i$ and $varphi_{ji}(g_j)=x$ and  the set of all indexes $j$ such $j\geq i$ and $varphi_{ji}(g'_j)=y$ both belongs to $\mathcal{F}$, and thus their intersection. But every $\varphi_{ji}$ is a homomorphism, so for every $j$ in the intersection $\varphi_{ji}(g_jg'_j)=xy$, and thus $\phi_i((g_j)(g'_j))=\phi_i(g_jg'_j)=xy$.  
	\end{proof}
	\begin{claim}
		The maps $\phi_i, \forall i\in I$ are compatible with $\{\varphi_{ji},j\geq i\}$.
	\end{claim}
	\begin{proof}
		Assume $j>i$ and $\phi_j(g_j)=x\in G_j$. Then it means that the set of al $k>j$ such that $\varphi_{kj}(g_k)=x$ belongs to $\mathcal{F}$. Since the maps $\varphi_{ji}$ are compatible, we get that $\varphi_{ki}(g_k)=\varphi_{ji}(x)$. So the set of all indexes $k\geq i$ which satisfy $\varphi_{ki}(g_k)=\varphi_{ji}(x)$, and thus $\phi_i(g_j)=\varphi_{ji}(x)=\varphi_{ji}(\phi_j(g_j))$. 
	\end{proof}
	Thus, the homomorphisms $\phi_i:\prod G_i\to G_i$ induce an homomorphism $\phi:\prod G_i\to \invlim \{G_i,\varphi_{ji}\}$.
	\begin{claim}
		$\phi$ is an epimorphism.
	\end{claim}
	\begin{proof}
		Recall that $\invlim {G_i, \varphi_{ji}}$ is in fact the subgroup of $\prod G_i$ of all elements $(g_i)$ such that for all $j>i$, $\varphi(g_j)=g_i$. So, let $(g_i)\in \invlim {G_i, \varphi_{ji}}$. We claim that $(g_i)\in \prod G_i$ is a source. Indeed, For every $j\geq i$, $\varphi_{ji}(g_j)=g_i$, and the set of all $j\geq i$ belongs to $\mathcal{F}$. Thus, $\phi_i((g_j))=g_i$ and $\phi((g_i))=g_i$.
	\end{proof}
	So, $\phi$ is the required epimorphism.
\end{proof}
\begin{proof}[Proof of Proposition \ref{non continuou inner epimorphism}]
Let $F_{\CC}(\mm)$ be a free pro-$\CC$ group. As we stated before, its local weight equals $\mm$, which means it has $\mm$ finite continuous quotients. Let $G=\prod G_i$ be the product of all its finite quotients. Notice that there are compatible epimorphisms between these groups which make $F_{\CC}(\mm)$ an inverse limit of $\invlim G_i$. By Lemma \ref{non continuous epimorphism from the product} there exists a noncontinuous epimorphism $\phi:\prod G_i\to  \invlim G_i\cong F_{\CC}(\mm)$. Now since $\omega_o(G)=\mm$ - which can be seen easily by forming $G$ as an inverse limit of finite products- $F_{\CC}(\mm)$ projects over $G$. Let $f=\psi\circ \varphi$ be the composition of these epimorphisms. The only thing left to show is that $f$ is noncontinuous. But this is immediate since $\varphi:F_{\CC}(\mm)\to G$ is a continuous epimorphism between profinite spaces, and thus a quotient map. 
\end{proof}

\bibliographystyle{plain}
\bibliography{references,../../Results_for_PHD/references}
\end{document}